\documentclass{amsart}
\usepackage{amsmath, amssymb}

\input xy
\xyoption{all}

\title[Operator algebras associated to modules]{Operator algebras associated to modules over an integral domain}

\author{Benton L. Duncan}

\address{Department of Mathematics\\
North Dakota State University\\
Fargo, North Dakota\\
USA}

\email{benton.duncan@ndsu.edu}



\begin{document}

\theoremstyle{plain}
\newtheorem{thm}{Theorem}
\newtheorem{lem}{Lemma}
\newtheorem{prop}{Proposition}
\newtheorem{cor}{Corollary}

\theoremstyle{definition}
\newtheorem{dfn}{Definition}
\newtheorem*{construction}{Construction}
\newtheorem*{example}{Example}

\theoremstyle{remark}
\newtheorem*{conjecture}{Conjecture}
\newtheorem*{acknowledgement}{Acknowledgements}
\newtheorem{remark}{Remark}

\begin{abstract} We use the Fock semicrossed product to define an operator algebra associated to a module over an integral domain. We consider the $C^*$-envelope of the semicrossed product, and then consider properties of these algebras as models for studying general semicrossed products. 
\end{abstract}

\maketitle

\section{Background on semicrossed products}

Although the study of semicrossed products (as objects of study) could be traced to \cite{Peters} and its precedents \cite{Arveson, ArvesonJosephson} the theory, as studied now was revived by \cite{DavidsonKatsoulis} where semicrossed products by $ \mathbb{Z}_+$ were shown to be complete invariants for topological dynamics.  This recent interest in semicrossed products as a class of operator algebras has seen significant recent growth, see for example \cite{Simple, Envelope, dilation, KK}. The primary focus of many of these new results is in the expansion of what semigroups are acting on the underlying $C^*$-algebra.  

In particular, in \cite{multivariable} an attempt was made to understand actions of the free semigroup on $n$-generators acting on a compact Hausdorff space. Other research focused on abelian semigroups, see \cite{DuncanPeters} and \cite{Fuller}. The paper \cite{DFK} formalized many of the different ideas used in the theory by combining constructions and different viewpoints into one overarching theme: studying the dilation theory of different classes of semicrossed products that are as much a function of the semigroup as the action on the underlying $C^*$-algebra. 

During this same time period the author in \cite{Duncan} introduced a construction, similar in spirit to \cite{CuntzLi} although different in perspective. Focusing on the non-selfadjoint theory this class of algebras was recognized as a semicrossed product \cite[Section 3]{Duncan}. Here we take the algebraic notion to the next stage. Here we consider a semicrossed product associated to a module over an integral domain $D$.  Of course  \cite{Duncan} is a special case of this as any ring is a module over itself.

Here we let a domain ``act'' on the group $C^*$-algebra associated to the module, via multiplication. As we view the group as a discrete group the actions are continuous and give rise to $*$-endomorphisms of the group $C^*$-algebras. For the most part we are undertaking this study to provide examples and motivation for questions in the broader area of semicrossed products.

For these purposes we first consider a general approach to operator algebras associated to a module over an integral domain, defining those properties that such an operator algebra should have.  We then justify the connection to semicrossed products by showing that an associated Fock semicrossed product will satisfy the list of requirements for an operator algebra associated to a module over an integral domain.  This allows us to identify the $C^*$-envelope of the algebra. 

We then use this connection to motivate and begin the study of questions for general semicrossed product and (often) the underlying dynamics. In the first such example we consider short exact sequences of modules and look at structures that arise in the operator algebra context. This leads us to propose looking at semicrossed products from a categorical point of view. Another example that we investigate is the notion of finite generation of a module over an integral domain. This gives rise to a similar notion of semicrossed products. We close with some results relating finite generation in semicrossed products to results on dynamics, in limited contexts. We also point out the difficulties inherent in considering this notion for arbitrary semicrossed products. We do, however, expect that these approaches will provide examples of results to look for in investigating more general semicrossed products.

In what follows we assume familiarity with the results and notation of \cite{DFK}.

\section{Operator algebras of modules over integral domains}

Let $R$ be an integral domain and $M$ be an $R$-module.  The multiplicative semigroup of $R$ (call it $R^{\times}$) acts on $M$ via the mapping $r(m) = r\cdot m$.  If we view $M$ as a topological group with the discrete metric then this action will induce an action of $R^{\times}$ on $C^*(M) = C(\widehat{M})$ (this latter is the continuous functions on the compact Pontryagin dual of $M$).  However this action is not necessarily as homeomorphisms and hence semicrossed products are a natural approach to understanding this action using operator algebras.  In particular the action $ \sigma_r: C^*(M) \rightarrow C^*(M)$ is invertible if and only if $r$ is invertible in $R$.

For this reason we define the operator algebra of a module over an integral domain to be the semicrossed product $ C^*(M) \times_{F} R^{\times}$.  In general one can consider a more general construction (which parallels the semicrossed product constructions of \cite{DFK}).  We first consider a specific representation.  Consider the Hilbert space $ \mathcal{H} := \ell^2(M) \otimes \ell^2(R^{\times})$.  For each $r \in R^{\times}$ we define $ S_r: \mathcal{H} \rightarrow \mathcal{H}$ on elementary tensors by $ S_r( v_m \otimes u_s) = v_{m} \otimes u_{rs}$ and then by extension we have $S_r \in B(\mathcal{H})$.  At the same time for each $m \in M$ we define $U^m(v_n \otimes u_r) = v_{mr+n} \otimes u_r$, and again extending to all of $ \mathcal{H}$.

We notice a few facts about the operators $S_r$ and $U^m$.  

\begin{prop} Let $M$ be an $R$ module and consider $m, n \in M$ and $r,s \in R^{\times}$.  Then the following are true:

\begin{enumerate}
\item \label{unitaries} $U^m$ is a unitary with $U^{-m} = (U^m)^*$. 
\item \label{isometries} $S_r $ is an isometry and if $r$ is invertible then $S_r$ is a unitary with $S_r^* = S_{r^{-1}}$.
\item \label{grouprepn} The map $ \mu: M \rightarrow B(\mathcal{H})$ induced by $ \mu(m) = U^m$ is a group representation with $ C^*(\mu(M)) \cong C_r^*(M)$ (the latter being the group $C^*$-algebra associated to $M$).
\item \label{semigrouprepn} The map $ \rho: R^{\times} \rightarrow B(\mathcal{H})$ given by $ \rho(r) = S_r$ is a semigroup representation.
\item \label{respectsmodule} The operators respect the module structure as follows: 
\begin{itemize} 
\item $U^mS_r = S_rU^{rm}$, 
\item $U^{m}S_{r+s} = S_{r+s}U^{rm}U^{sm}$, 
\item $U^{m+n}S_r = S_rU^{rm}U^{rn}$, and
\item $S_1 = U^0= 1_{\mathcal{H}}$, 
\end{itemize}
\end{enumerate} \end{prop}

\begin{proof}
\begin{enumerate}

\item We calculate that \[ \langle U^n(v_m \otimes u_r), v_l \otimes u_s \rangle = \begin{cases} 1 & r=m+nr \mbox{ and } r=s \\ 0 & \mbox{ otherwise} \end{cases}  \] and that \[ \langle v_m \otimes u_r, U^{-n} (v_l \otimes u_s) = \begin{cases} 1 & m = l-ns \mbox{ and } s = r \\ 0 & \mbox{ otherwise} \end{cases}.\]
It follows using linearity that $(U^n)^* = U^{-n}$ and then calculating we see that $U^nU^{-n} = 1_{B(\mathcal{H}} = U^{-n}U^n$ which tells is that each $U^n$ is a unitary.

\item We begin by calculating the adjoint of $S_r$.  To do this we first write $H_r$ to be the subspace of $\mathcal{H}$ spanned by elements of the form $ v_m \otimes u_{rt}$ where $ m \in M $ and $ t \in R^{\times}$. We define $T_r( v_m \otimes u_{rt}) = v_m \otimes u_t$ on $H_r$ and $T_r $ is zero on the orthgonal complement of $H_r$.

Now \begin{align*} \langle S_r( v_m \otimes u_s), v_n \otimes u_t \rangle & = \langle v_m \otimes u_{sr}, v_n \otimes u_t \rangle \\ & = \begin{cases} 1 & \mbox{ if } t = sr \\ 0 & \mbox{ otherwise} \end{cases} \end{align*} and similarly we have \[ \langle  v_m \otimes u_s, T_r(v_n \otimes u_t) \rangle = \begin{cases} 1 & \mbox{ if } t = rs \\ 0 & \mbox{ otherwise} \end{cases} \] and hence $T_r$ is the adjoint of $S_r$. 

We then can calculate that $ T_rS_r ( u_n \otimes v_s) = T_r(u_n \otimes v_{rs}) = u_n \otimes v_s$ and hence $S_r$ is an isometry with range equal to $H_r$.  Notice that if $r$ is invertible then for any $ m \in M$ and $ t \in R^{\times}$ then $ v_m \otimes u_t = v_m \otimes u_{rr^{-1}t}$ and hence $H_r = \mathcal{H}$ which tells us that $S_r$ is surjective and hence is a unitary with $T_r$. In this case a simple calculation will show us that $S_{r^{-1}} = T_r$.

\item
Notice that this representation induces a unitary representation of $M$ acting on $B( \mathcal{H})$.  Also if we consider the subspace spanned by $ v_m \otimes u_1$ for all $ m \in M$ then the restriction of our representation to this subspace will give the left regular representation of $M$ acting on a Hilbert space isomorphic to $\ell^2(M)$. Since $M$ is abelian and hence amenable it follows that the representation must be a faithful representation of $C^*(M)$.

\item
This follows by noting that $ S_{rs} (v_m \otimes u_t) = (v_{r} \otimes u_{rst} ) = S_r( v_{m} \otimes u_{st}) = S_r(S_s)(v_m \otimes u_t)$ and hence $S_{rs} = S_rS_s$ yielding a semigroup representation.  

\item
We verify the first property. First we see that \begin{align*} U^mS_r(v_n \otimes u_s) & = U^m( v_{n} \otimes u_{sr})  \\ & = v_{n+mrs} \otimes u_{sr} \\ & = S_r(v_{n+mrs} \otimes u_s) \\ & = S_rU^{rm} (v_n \otimes u_s )\end{align*} and then extending to all of $B( \mathcal{H})$ we have the indicated identity.  The second and third properties here follow from the first. The final is a simple calculation.
\end{enumerate}

\end{proof}

Given a Hilbert space $H$ and two collections of operators $ \mathcal{U} = \{ u^m: m \in M \}$ and $ \mathcal{S} = \{ s^r: r \in R^{\times} \}$ we define a pair of maps $\mu: M \rightarrow B(\mathcal{H})$ by $ \mu(m) = u^m$ and $ \rho: R^{\times} \rightarrow B(\mathcal{H})$ by $ \rho(r) = s_r$.  We say that the pair $ (\mu, \rho)$ is an isometric representation of $M$ with respect to $R$ if (1-5) of the above proposition are satisfied for the collection $ \mathcal{U}$ and $ \mathcal{S}$, it is said to be unitary if the family $ \mathcal{S}$ consists of unitaries.  The previous proposition gives us a canonical isometric representation of $M$ with respect to $R$, which we call the Fock representation. 

As in \cite{DFK} we can consider other semicrossed products associated to the action of $ R^{\times}$ acting on $C^*(M)$.  In particular there are three standard operator algebras associated to such an action: the Fock algebra which we denote $C^*(M) \times_F R^{\times}$, the isometric semicrossed product which we denote $C^*(M) \times_i R^{\times}$ and the unitary semicrossed product which we denote by $C^*(M) \times_u R^{\times}$.  

The algebra $C^*(M) \times_F R^{\times}$ is the norm closed algebra inside $B(\ell^2(M) \otimes \ell^2(R^{\times}))$ generated by the families $ \{ S_r \}$ and $\{ U^m \}$.  This algebra will, for the most part, be the focus of this paper due to its concrete realization as acting on a Hilbert space.  The algebras $C^*(M) \times_i R^{\times}$ and $C^*(M) \times_u R^{\times}$ can be constructed by considering the algebra $A_0 $ consisting of finite sums of the form $ \sum s_ra_r$ such that $a_r \in C^*(M)$ with a convolution multiplication that respects $a s_r= s_r\alpha_r(a)$, where $\alpha_r$ is the action on $C^*(M)$ induced by $r$ acting on $M$.  One notices that any isometric (unitary) representation of $M$ with respect to $R$ gives rise to a representation of the algebra $A_0$ acting on the same Hilbert space.  It follows that one can then norm $A_0$ by taking the supremum over all isometric (unitary) representations of $M$ with respect to $R$.  Completing the algebras with respect to the induced norms gives rise to the semicrossed products $C^*(M) \times_i R^{\times}$ and $C^*(M) \times_u R^{\times}$, respectively.  

\begin{prop} The Fock semicrossed product $C^*(M) \times_F R^{\times}$ is isomorphic to the unitary semicrossed product $C^*(M) \times_u R^{\times}$ if and only if $M$ is torsion free. \end{prop}

\begin{proof}
The forward direction of this result follows from \cite[Theorem 3.5.4]{DFK} by noting that if $M$ is torsion free then the action of $r$ on $M$ is injective for every $ r \in R^{\times}$.  Hence $R^{\times}$ acts injectively on $C^*(M)$ and the result applies.  For the alternative, again from \cite{DFK} it is noted that if the action of $R^{\times}$ is not injective (i.e. has torsion) then $C^*(M) \times_u R^{\times}$ does not contain $C^*(M)$ but rather only a quotient of $C^*(M)$.  However $C^*(M) \times_F R^{\times}$ always contains a copy of $C^*(M)$ and hence the two algebras are not isomorphic.
\end{proof}

It is not the case that the isometric semicrossed product is isomorphic to the other two semicrossed products, even in the case that the module is torsion free, since the action of $R^{\times}$ on $M$ acts as automorphisms if and only if $R$ is a field.  In this latter case \cite[Theorem 3.5.6]{DFK} we get that these algebras are isomorphic.  In this case the algebra $C^*(M) \times_F(R^{\times})$ is a crossed product by an abelian (and hence amenable) group which means that the universal crossed product and the reduced crossed products are isomorphic.  When $R$ is not a field it is straightforward to see that $C^*(M) \times_F R^{\times}$ is not a $C^*$-algebra and hence one important consideration is the $C^*$-algebra it generates.  There is a canonical $C^*$-algebra in which an operator algebra embeds, this algebra is the $C^*$-envelope and is thought of as the ``smallest'' $C^*$-algebra which the nonselfadjoint algebra generates.  Following \cite{DFK} we consider the $C^*$-envelopes of the algebra $C^*(M) \times_F R^{\times}$ in the case that $M$ is torsion free.   The only added complexity is verifying that the construction still produces an $R$-module.

\begin{thm} Let $R$ be an integral domain and $M$ be an $R$-module.  There exists a $Q(R)$-module $N$ and an $R$-module injection $i: M \rightarrow N$ such that $C^*(N) \times_F Q(R)^{\times}$ is the $C^*$-envelope of $C^*(M) \times_F R^{\times}$ \end{thm}

\begin{proof}
We mimic the construction from \cite[Section 3.2]{DFK} working at the level of $R$-modules rather than $C^*$-algebras.

We note first that $Q(R)^{\times}$ is the enveloping group for $R^{\times}$. We put a partial ordering on $R^{\times}$ given by $s \leq t$ if in $Q(R)^{\times}$ we have that $ts^{-1} \in R^{\times}$. Then for $r \in R^{\times}$ we let $M_r = M$ and define connecting maps $ \alpha_{t,s}: M_s \rightarrow M_t$ when $ s \leq t$ by $ \alpha_{t,s} (m) = ts^{-1}(M)$. Then $N:= \lim_{\rightarrow} (M_r, \alpha_{t,s})$ is an $R$-module such that the $r$ action is now surjective and injective on $N$, hence we can define the $r^{-1}$ action on $N$ by inverting the $r$-action giving an action of $Q(R)^{\times}$ on $N$.

Noting that the functor from abelian groups to $C^*$-algebras preserves direct limits it follows that $C^*(N)$ is the algebra constructed in \cite{DFK} and the result now follows from \cite[Theorem 3.2.3]{DFK}.
\end{proof}

Since the $C^*$-algebras in our semicrossed products arise from group $C^*$-algebras we can use information about group $C^*$-algebras to improve our analysis.  In the following section we consider some results that will be helpful.

\section{Some useful results on groups, semigroups, and semicrossed products}

Let $M$ and $N$ be discrete abelian groups with $N$ a normal subgroup of $M$.  This relationship induces a short exact sequence \[ 1 \rightarrow N \rightarrow M \rightarrow M/N \rightarrow 1. \]

While it is not the case that a short exact sequence of groups gives rise to a short exact sequence of the associated $C^*$-algebras, the relationship between groups does give us a relationship between the associated $C^*$-algebras.  We build up to the result with some preliminary lemmas, the first is a combination of Propositions 2.5.8 and 2.5.9 of \cite{BrownOzawa}.

\begin{lem} Given an inclusion of groups $N \subseteq M$ there is an inclusion of $C^*$-algebras $ C^*(N) \subseteq C^*(M)$ and $C^*_r(N) \subseteq C^*_r(M)$. \end{lem}

In the reduced case the proof involves noticing that $ \ell^2(N)$ is a subspace of $\ell^2(M)$ and the left regular action of $M$ acting on $\ell^2(M)$ gives rise to the left regular action of $N$ acting on the subspace $\ell^2(N)$.  Hence the reduced $C^*$-algebra $C_r^*(N)$ can be viewed as sitting inside $C_r^*(M)$. The proof for the universal algebra is more subtle. We refer the reader to \cite{BrownOzawa} for the details.

\begin{lem} Given a surjection of groups $\pi: M \rightarrow M/\ker\pi$ there is a unital surjective $*$-homomorphism $\pi: C^*(M) \rightarrow C^*(M/ \ker \pi)$. \end{lem}

\begin{proof} This is a straightforward application of the universal property of $C^*(M)$ (i.e.\ any representation of $C^*(M/ \ker \pi)$ induces a representation of $C^*(M)$ via composition with the map $ \pi$). \end{proof}

Notice that if we know that the groups are amenable then the reduced $C^*$-algebra will satisfy the universal property as in the preceding lemma and hence both results will apply.  

Given a $C^*$-algebra $A$ with $C^*$-subalgebra $B$ we say that a unital representation $\pi: A \rightarrow C$ trivializes $B$ if $B = \pi^{-1}(1_C)$.  We consider this in the context of an exact sequence of groups.

\begin{thm}\label{exact} Let $1 \rightarrow N \rightarrow M \rightarrow M/N \rightarrow 1$ be a short exact sequence of discrete abelian groups.  Then the natural surjection $\pi: C^*(M) \rightarrow C^*(M/N)$ trivializes the subalgebra $C^*(N)$ inside $C^*(M)$. \end{thm}

\begin{proof}
By definition the generators for $C^*(N)$ inside $C^*(M)$ will be mapped to $1_{C^*(M/N)}$ by the natural surjection and hence the range of $C^*(N)$ under the natural surjection is $\{ \mathbb{C} 1_{C^*(M/N)} \}$.  

Now consider $U_g \in C^*(M)$ and consider $ \pi(U_g)( \zeta_{xN})$.  If $ \pi(U_g)$ is a multiple of the identity then $ \zeta_{xN} = \lambda \zeta_{gxN}$ for some $ \lambda$.  It follows that $ \lambda$ must equal $1$ and $gxN = xN$ for all $x$.  Then there is some $n \in N$ such that $gx = xn$.  Canceling the $x$ tells us that $ g \in N$ and hence the only $U_g$ which $\pi$ maps to the identity are those inside $C^*(N)$.

It follows that any polynomial in the generators of $C^*(M)$ the same result will apply.  Specifically $ \sum_{i=1}^n \alpha_{g_i}U^{g_1} \mapsto \mathbb{C}$ implies that $ g_i \in N$ for all $i$ and hence by continuity $\pi^{-1}(\mathbb{C} 1_{C^*(M/N)}) = C^*(N)$.
\end{proof}

\begin{remark} In the preceding proof it is not necessary that the groups be abelian (amenable would suffice).  However, for our purposes abelian will be more than enough. \end{remark}

Consider the case of $C^*(\mathbb{Z}) \cong C(\mathbb{T})$.  Notice that any subgroup of $ \mathbb{Z}$ is of the form $ m\mathbb{Z}$.  Then $ \mathbb{Z} / m \mathbb{Z} \cong \mathbb{Z}_m$ and there is a natural representation $\pi:  C(\mathbb{T}) \rightarrow \oplus_{m} \mathbb{C}$ given by $\pi(f(z)) = ( f(e^0), f(e^{\frac{2\pi i}{m}}), f(e^{\frac{ 4 \pi i }{m}}), \cdots , f(e ^{ \frac{(2m-2) \pi i}{m}}))$.  Notice that $ \pi(z^k) = (1,0,0, \cdots, 0)$ if and only if $ m|k$.  Then if we set $A = \{ f: \pi(f) = ( \lambda, 0, 0, \cdot, 0 ) \}$ then $A$ is a $C^*$-subalgebra of $C(\mathbb{T})$ generated by $ \{ z^km: k \in \mathbb{Z} \}$.  Notice that $\pi$ trivializes $A$ and in fact $A \cong C^*(m \mathbb{Z})$. 

On the other hand notice that the representation $ \pi: C(\mathbb{T}) \rightarrow \mathbb{C}$ given by $ \pi(f(z)) = f(i)$ trivializes the $C^*$-subalgebra generated by $ \{ z^4 \}$.  However, $ \pi^{-1}(\mathbb{C} 1) )= C^*(\mathbb{Z})$.

In fact we can go further.  If $K $ is a subgroup of $G$ then there is a universal trivializing algebra through which all trivializing representations for $C^*(K) \subseteq C^*(G)$ factor.  This representation is, in fact, the standard representation $ \pi: C^*(G) \rightarrow C^*(G/K)$.

\begin{prop} \label{trivializing} Let $A$ be a $C^*$-subalgebra of $C^*(G)$.  There is a subgroup $K \subset G$ such that $A \cong C^*(K)$ if and only if \[ A = \bigcap \{\pi^{-1}( \mathbb{C}1): \pi \mbox{ trivializes }C^*(K) \} .\]
\end{prop}

\begin{proof} The forward direction is described before the statement of the proposition. We consider here the converse. So let $\pi: C^*(G) \rightarrow B(\mathcal{H})$ be a representation which trivializes $A$. Now let $N_{\pi} = \{ g \in G: \pi(U^g) = 1 \}$. Checking we see that $N_{\pi}$ is a subgroup of $G$, since if $g, h \in N$ then $\pi(U^{-g}) = \pi((U^{g})^*) = \pi(U^g)^* = 1^* = 1$ and $ \pi(U^{g+h}) = \pi(U^gU^h) = \pi(U^g)\pi(U^h) = 1$ so that $N_{\pi}$ is closed with respect to inverses and the group operation. Then we define $K = \cap \{ N_{\pi}: \pi \mbox{ trivializes } A \}$. 

It remains to prove that $A = C^*(K)$. Certainly we know that $C^*(K)$ can be viewed as sitting inside $A$. Next we consider the natural map $ \pi: C^*(G) \rightarrow C^*(G/K)$ which trivializes $C^*(K)$ and hence $A \subseteq C^*(K)$.
\end{proof}

Next we consider similar results for semicrossed products.  

\begin{prop} \label{invariance} Let $A$ be a closed subalgebra of $B$ and let $S$ be a discrete abelian semigroup acting on $B$ via completely contractive endomorphisms. If $\sigma_s(A) \subseteq A$ for all $ s \in S$ then $A\times_{F} S \subseteq B \times_{F} S$. \end{prop}

\begin{proof} 
Let $ \pi: B \rightarrow B(\mathcal{H})$ be a faithful representation, then notice that $ \pi|_A: A \rightarrow B(\mathcal{H})$ is a faithful representation. Further the canonical Fock representation of $B \times_F S \rightarrow B(\mathcal{H} \otimes \ell^2(S))$ restricts to, since $\sigma_s(A) \subseteq A$ for all $s \in A$, a Fock representation of $A \times_F S$ acting on $ B(\mathcal{H} \otimes \ell^2(S))$ and the result now follows, since this representation will be unitarily equivalent to the standard canonical Fock representation of $A \times_F S$.
\end{proof}

\begin{prop} Let $\pi: B \rightarrow C$ be a completely contractive quotient map and assume that an abelian discrete semigroup $S$ acts on $B$ (via $\sigma$) and on $C$ (via $\tau$) as completely contractive endomorphisms.  If $ \tau_s \circ \pi = \pi \circ \sigma_s$ for all $s$ then there is a completely contractive quotient map $\tilde{\pi}: B \times_F S \rightarrow C \times_F S$ such that $ \tilde{\pi}(b) = \pi(b)$ for all $b \in B$ and $\tilde{\pi}(S_s) = T_s$ for all $s \in S$. \end{prop}

\begin{proof}
This proof is a simple outcome of the Fock construction since the map $\pi: B \rightarrow C$ with the associated representation of $S_s \mapsto T_s$ gives rise to a Fock representation of the pair $(B, \sigma)$. By definition \cite[Definition 3.5.1]{DFK} the Fock algebra $B \times_F S$ is universal for Fock representations we get the indicated homomorphism.
\end{proof}

Let $X$ be a compact Hausdorff space and consider an abelian semigroup $S$ acting on $X$ by continuous maps $ \{\sigma_s: s \in S \}$.  If $S = M \times N$ then there is a natural action of $N$ on $C(X) \times_F M$ given on the generators of $C(X) \times_FM$ by $\sigma_n(f(x)) = f( \sigma_n(x))$ and  $\sigma_n(S_m) = S_m$ for all $m \in M$; it is straightforward to that verify the covariance relationship is preserved.  

\begin{thm} Let $S = M\times N$ be a direct product of abelian semigroups acting via $*$-endomorphisms on a $C^*$-algebra $A$. Then $A \times_F S \cong ( A \times_F M) \times_F N \cong ( A \times_F N) \times_F M$. \end{thm}

\begin{proof} For $s \in S$ we denote by $ \sigma_s: A \rightarrow A$ the associated $*$-endomorphism. For $ n \in N$ we define an action on the generators of $A \times_F M$ by $ \alpha_n ( a) = \sigma_n(a)$ for all $ a \in A$ and $ \alpha_n(S_m) = S_m$ for all $ m \in M$. This yields a Fock representation of $A\times_F M$ since the actions of $M$ and $N$ on $A$ commute and hence it induces a completely contractive representation of $A \times_F M$.

One then considers the canonical Fock representation of $(A \times_F M) \times N$ and notes that this yields a Fock representation of $A \times_F (M\times N)$.  Similarly the canonical Fock representation of $A \times_F (M \times N)$ yields a Fock representation of $(A \times_F M) \times_F N$. A similar result holds for the alternative with $M$ and $N$ changing place and hence the result holds.
\end{proof}

\begin{remark} Clearly this extends via induction to any finite product of abelian semigroups. \end{remark}

\begin{remark} Notice that this is true (as a special case) for a crossed product by the product of two abelian groups, since abelian groups are amenable and hence the canonical Fock representation is faithful in this context. \end{remark}

As an application of this result, if we denote by $R_u$ the group of units in $R$ then one can see that $C^*(M) \times_F R^{\times} \cong (C^*(M) \times_F R^u) \times (R^{\times}/R^u)$.  This allows us to identify (as in \cite{Duncan}) the diagonal of the algebra $C^*(M) \times_FR^{\times}$ as the crossed product $C^*(M) \times R^u$.

In our extended context however submodules are more than just subgroups of $M$.  We thus extend the previous results in the following section.

\section{Submodules and Quotients}

In \cite[Corollary 2.5.12]{BrownOzawa} it is shown that if $ N$ is a subgroup of $M$ then there is a conditional expectation $E: C^*(M) \rightarrow C^*(N)$ such that for $ m\in M$ we have $E(U^m) = \begin{cases} U^m: m \in N \\ 0 \end{cases}$ (this is true for both the universal and reduced $C^*$-algebras). A partial converse of this can be found as a corollary of results in \cite{Choda, LOP}; the former containing a version of the result for the reduced $C^*$-algebra for a discrete group and the latter containing a version of the result for abelian group $C^*$-algebras (in which the universal and reduced $C^*$-algebras are isomorphic). Specifically, we have the following proposition.

\begin{prop} Let $A \subseteq C^*(M)$ be a subalgebra. Then there is a subgroup $N \subseteq M$ such that $A \cong C^*(N)$ if and only if there is a conditional expectation $E: C^*(M) \rightarrow A$ with for any $ m \in M$ $E(U^m) = \begin{cases}  U^m &: U^m \in A \\ 0 & \mbox{ otherwise}. \end{cases}$. \end{prop}

Let $R$ be a ring and $M$ be an $R$-module.  A submodule $N \subseteq M$ is a subgroup of $M$ which is closed under the action of $R^{\times}$.  To study submodules and quotient modules in the context of operator algebras we use the results of the preceding section and the preceding result in the context of the Fock algebra associated to a module over an integral domain. The following is just a combination of the preceding result and Proposition \ref{invariance}.

\begin{prop} Let $A$ be a subalgebra of $C^*(M) \times_F R^{\times}$. There is a submodule $N \subseteq M$ with $A \cong C^*(N) \times_F R^{\times}$ if and only if there is a conditional expectation $E: C^*(M) \rightarrow A \cap C^*(M)$ such that $ E(U^m) = \begin{cases} U^m & U^m \in A \\ 0 & \mbox{ otherwise} \end{cases}$ and $\sigma_s(A) \subseteq A$ for all $ s \in S$. \end{prop}

We look now to find an alternative that does not necessarily require the latter condition. Specifically assume that $A$ is a subalgebra of $C^*(M) \times_F R^{\times}$ and that $N$ is a subgroup of $M$. Then we consider the Hilbert space $ \mathcal{H}_N:=\ell^2(M/N) \otimes \ell^2(R^{\times}$, which we call the quotient Hilbert space for the pair $(M,N)$. For every $m \in M$ we define $U^m: \mathcal{H}_N \rightarrow \mathcal{H}_N$ on elementary tensors by $U^m (v_{g+N} \otimes u_r) = v_{(mr+g) +N} \otimes u_r$ and extending by linearity to all of $B( \mathcal{H}_N)$. Similarly for $t \in R^{\times}$ we define $S_t(v_{g+N} \otimes u_r) =v_{g+N} \otimes u_{rt}$ and extending to all of $B(\mathcal{H}_N)$. We call the collection of operators $\{ U^m, S_r \}$ the quotient operators for the pair $(M,N)$.

\begin{thm}
Let $N$ be a subgroup of the $R$-module $M$.  Then $N$ is a submodule if and only if the quotient operators for the pair $(M,N)$ give rise to a covariant representation of $C^*(M) \times_F R^{\times}$.
\end{thm}

\begin{proof}
If $N$ is a submodule for $N$ then the quotient operators are acting on the canonical Hilbert space for $C^*(M/N) \times_F R^{\times}$ and hence they will give rise to a covariant representation of $C^*(M) \times_F R^{\times}$, via the quotient map $q: C^*(M) \rightarrow C^*(M/N)$.

For the reverse direction notice that $N$ is an $R$-submodule if and only if $rn \in N$ for every $r \in R^{\times}$ and $ n \in N$. So, if $M$ is not an $R$-submodule then there must be some $r \in R^{\times}$ and $ n \in N$ such that $rn \not\in N$.  Notice that $U^n = 1$ by definition but $U^{rn} \neq 1$ but the covariance condition would imply that $S_rU^{rn} = U^{n} S_r$ for any $r$.  This would force the conclusion that $S_r U^{rn} = S_r$ but since $S_r$ is an isometry it follows that $U^{rn} = 1$ which is a contradiction.
\end{proof}

\begin{remark} As was noted in the proof notice that if $N$ is an $R$-submodule of $M$ then $C^*(N)$ is trivialized by the representation induced by the quotient operators.
\end{remark}

We then have the simple corollary concerning quotient modules.

\begin{cor}
Let $M$ be an $R$-module and $\pi: C^*(M) \times_F R^{\times} \rightarrow A \times_F R^{\times}$ be a completely contractive homomorphism.  Then $A \cong C^*(M/N)$ for some $R$-submodule $N$ if and only if $\pi^{-1}(\mathbb{C}) = C^*(N)$. 
\end{cor}

\begin{proof} 
The forward direction is just an implication of Theorem \ref{exact}. For the backward direction notice that once we know that $\pi^{-1}(\mathbb{C}) = C^*(N)$ then we only need to verify that the quotient operators induce a completely contractive representation. But that follows from the fact that $ \pi$ is a completely contractive representation of $C^*(M) \times_F R^{\times}$.
\end{proof}

This example now motivates similar concepts for general semicrossed products. Given two semicrossed products $C(X) \times_{\mathcal{F}} S$ and $C(Y) \times_{\mathcal{F}} S$ which are generated by $ \{ C(X) \} \cup \{ S_t \}$ and $\{ C(Y) \} \cup \{ T_t \}$. We say that the latter is a quotient semicrossed product of the former if there is a covariant representation $ \pi: C(X) \times_{\mathcal{F}} S \rightarrow C(Y) \times_{\mathcal{F}} S$ which is surjective and satisfies $\pi(S_t) = T_t$ for all $t \in S$. In effect we are requiring a surjective $*$-homomorphism $ \pi: C(X) \rightarrow C(Y)$ such that $ T_t\pi( f(x)) = \pi(\alpha_t(f(x))) T_t$ for all $t \in S$.

Next we say that $C(Y) \times_{\mathcal{F}}S$ is a sub-semicrossed product of $C(Z) \times_{\mathcal{F}} S$ if there is a quotient semi-crossed product $C(X) \times_{\mathcal{F}} S$ such that the composition     sends $C(Z)$ to the identity in $C(X)$ and fixes the semigroup generators.

Using Gelfand theory we can readily see that the following are required for the existence of a sub-semicrossed product.

\begin{itemize}
\item A continuous injection $ \pi: X \rightarrow Z$.
\item A continuous surjection $\tau: Y \rightarrow X$.
\item A continuous surjection $\sigma: X \rightarrow N$ where $N$ is the zero set of the ideal $J$ corresponding to the quotient $C(Y) /J \cong C(X)$.
\item The continuous map $\pi \circ \tau$ must be a constant map.
\item The action of $S$ commutes with $\pi$ and $ \tau$ (i.e.\ $\pi \circ \alpha_s = \beta_s \circ \tau$ for all $s$.
\end{itemize}
 
Notice that the construction for submodules and quotient module operator algebras are special cases of this construction for general semi-crossed products.

Although the above list gives necessary conditions there is no claim that these conditions are sufficient. A useful categorical approach to this question would necessitate a complete set of necessary and sufficient conditions. Such an approach would also likely require considering the different universal semicrossed products to give rise to a reasonable theory.

\section{Finite generation}

One view of the collection of algebras considered in this paper is as a class of readily understood semicrossed products that can motivate questions in the general case of semicrossed products.  We consider an example of this type of process in this section.  We will focus on finitely generated modules over an integral domain $R$. We remind the reader that a module $M$ is finitely generated as an $R$-module if there exists some $m_1, m_2, \cdots, m_n \in M$ such that for every $m \in M$ there is $r_1, r_2, \cdots, r_n \in R$ such that $ m = r_1m_1 + r_2m_2 + \cdots + r_nm_n$.  The set $ \{ m_1, m_2, \cdots, m_n \}$ is called a generating set for $M$.  In the case that $M$ has a generating set of size $1$ we say that $M$ is cyclic.

We now consider how this translates to a property of operator algebras, specifically we will focus first on the $C^*$-algebra generated by $C^*(M) \times_F R^{\times}$ inside $B(\ell^2(M) \otimes \ell^2(R^{\times}))$ to motivate the general definition. 

If $M$ is cyclic as an $R$-module then there is some $m \in M$ such that for any $n \in M$ there is $r \in R$ such that $rm = n$.  Inside the $C^*$-envelope of the semicrossed product we then have that $S_r^* U_mS_r = U^n$.  It follows that $ C^*(M)$ is the closure of the unitization of the subspace generated by $ \{ S_r^*U^mS_r: r \in R^{\times} \}$.  We can take this as a general definition for a semicrossed product. Let $C(X) \times_{{\mathcal{F}}} S$ be a Fock algebra associated to an action of $S$ (an abelian discrete semigroup) on $C(X)$ acting as completely contractive endomorphisms.  We will call the the closure of the unitization of the subspace spanned by $\{ S_r^*f(x)S_r: r \in R^{\times} \}$ the cyclic subspace generated by $ f$ and will denote it by $ \langle \langle f \rangle \rangle$. We can say that the action is cyclic if there is some $f \in C(X)$ such that $C(X) = \langle \langle f \rangle \rangle$.

\begin{example} Let $\sigma: X \rightarrow X$ be the identity map.  Then $ S = \mathbb{Z}_+$ and $S$ acts on $C(X)$ by iterating the map $\sigma$. In this case $C(X) \times S$ is cyclic if and only if $X$ consists of one or two points.  To see this notice that that for any $f(x)$ we have $ S_n^*f(x) S_n = f(x)$ for all $n$ and hence \[ \langle \langle f \rangle \rangle = \begin{cases} {\rm span} \{ 1, f(x) \}  \cong \mathbb{C}^2 &: \mbox{ if } f(x) \neq \lambda\cdot 1 \mbox{ for any } \lambda \in \mathbb{C} \\ \mathbb{C} & \mbox{ otherwise} \end{cases}. \] Which happens if and only if $X$ consists of at most two points.
\end{example}

\begin{example} Let $X = \{ 0, 1, \cdots, n-1 \}$ and $ \sigma: X \rightarrow X$ be given by $\sigma(i) = \begin{cases} i+1 &: i \leq n-2 \\ 0 &: i = n-1 \end{cases}$.  Let $\mathbb{Z}_+$ act on $X$ by iterating $\sigma$.  Then $C(X) \times S$ is cyclic since $ f(x) = \chi_{1} \in C(X)$ and $S_r^*f(x)S_r = \chi_{r \mod n}$ so that $C(X) = {\rm span} \{ S_r^*f(x)S_r: 1 \leq r \leq n \} \subseteq \langle \langle f(x) \rangle \rangle$.
\end{example}

This points to a general result. We denote by $\mathcal{O}(x) = \overline{ \{ \sigma_n(x) : n \in \mathbb{Z}_+ \}}$.

\begin{prop} Let $X$ be a totally disconnected compact set and $\sigma: X \rightarrow X$ be a continuous map.  Then if there is some $z$ such that $ X \setminus \mathcal{O}(z) \subseteq \{ y \}$ for some point $y$ then $C(X) \times \mathbb{Z}_+$ is cyclic. \end{prop}

\begin{proof}
As $X$ is totally disconnected it is generated by projections of the form $ \chi_{\{ x \}}$ such that $ x \in X$.  Let $f(x) = \chi_{\{ z \}}$ then $ S_n^*f(x)S_n = \chi_{\{ \sigma^n(z) \}}$ and notice that ${\rm span} \{ S_n^*f(x)S_n \}$ is a subalgebra of $C(X)$. The unitization of this algebra separates the points of $X$ and is unital and hence by the Stone-Weierstrass Theorem we have that $\langle \langle f(x) \rangle \rangle = C(X)$.
\end{proof}

Considering arbitrary compact sets $X$ points out two problems that can arise for a fixed $f$: 
\begin{itemize}
\item $\langle \langle f \rangle \rangle$ need not be an algebra.
\item $\langle \langle f \rangle \rangle$ need not be self-adjoint.
\end{itemize}
If we know a-priori these two facts then the semicrossed product is cyclic if and only if the set $ \langle \langle f \rangle \rangle$ separates points (via the Stone-Weierstrass Theorem).  In general this is not going to be true.

\begin{example}
Let $X = [-1,1]$ and $\sigma(t) = t^2$, and $f(t) = t$. Then $\langle \langle f \rangle \rangle$ is not an algebra since $t^3 = f^3$ but $f(\sigma^n)$ is an even function for all $ n \geq 1$. Now calculus tells us that $t^3 \neq \lambda_1 + \lambda_2 t + g(t)$ for any even function $g(t)$ since the second derivative of $t^3 $ is an odd function but the second derivative of the latter term is an even function. It follows that $ \langle \langle f(t) \rangle \rangle$ is not an algebra. 
\end{example}   

Of course, for any cyclic module over a commutative ring $R$, the associated semicrossed product algebra is cyclic (by definition).  
It follows that one can easily construct examples of cyclic semicrossed products.  The reason this appears to work, however, is that the semigroup is ``large'' which means that $C^*(M)$ will ``move'' a lot when we consider all actions of $R^{\times}$ on it.

The general case of $n$-generation is similar.  We translate the commutative algebra definition into the context of $C^*(M) \times_{\mathcal{F}} R^{\times}$. Then we tweak the definition into the more general semicrossed product context.  For brevity we will proceed right to the general definition. We denote by $ \langle \langle f_1, f_2, \cdots, f_n \rangle \rangle$ to be the closure of the span of all elements of the form $ g_1\cdot g_2\cdot  \cdots \cdot g_n$ where $ g_i \in \langle \langle f_i \rangle \rangle$. Notice that since each of the sets $ \langle \langle f_i \rangle \rangle$ is unital then so is $ \langle \langle f_1, f_2, \cdots, f_n \rangle \rangle$ and we also have that $ \langle \langle f_i \rangle \rangle \subseteq \langle \langle f_1, f_2, \cdots, f_n \rangle \rangle$ for all $i$.

\begin{prop} If $X$ is finite then $C(X) \times S$ is finitely generated by no more than $|X|$ elements of $C^*(X)$. \end{prop}

\begin{proof}
This requires no action by $S$ since is $X = \{ x_1, x_2, \cdots, x_n \}$ then $C(X)$ is spanned by $ f_i := \chi_{\{ x_i \}}$ and hence $C(X) = \langle \langle f_1, f_2, \cdots, f_n \rangle \rangle$.
\end{proof}

We can often do with fewer generating elements. 

\begin{example} 
Let $X$ be finite then $C(X) \times \mathbb{Z}_+$ is finitely generated by $k$ elements where $k$ is the number of components in the orbit of the action. To see this, for each component of the orbit of $\sigma$ choose a single element $x_i$.  Then if we set $ g_i = \chi_{ \{ x_i \}}$ then $\langle \langle g_i, \rangle \rangle$ is the unitization of $\{ f(x): \mbox{ the support of } f(x) \subseteq \mathcal{O}(x_i ) \}$. The result now follows.
\end{example}

Notice that these constructions work for any representation which is faithful on $C^*(M)$ and sends the $S_g$ to an isometry satisfying the covariance conditions.

We have the following proposition.

\begin{prop} Assume that $C(X) \times_{\mathcal{F}} \times S$ is finitely generated and there is a surjective homomorphism $\pi: C(X) \times_{\mathcal{F}} S \rightarrow C(Y) \times_{\mathcal{F}}  S$.  Then $C(Y) \times_{\mathcal{F}} S$ is finitely generated. \end{prop}

\begin{proof}
This relies on the fact that $ \langle \langle \pi(f_1), \pi(f_2), \cdots, \pi(f_n) \rangle \rangle = \pi ( \langle \langle f_1, f_2, c\dots, f_n \rangle \rangle) = C(Y)$ if $\pi$ is surjective.
\end{proof}

It is well known \cite[Chapter 4]{Roman} that finite generation is not inherited by submodules.  We do not, however, know that a module is finitely generated if the associated semicrossed product is finitely generated.  If we has such a result then we would have an example where a sub-semicrossed product does not inherit the finitely generated property. 

\bibliographystyle{plain}

\begin{thebibliography}{00}

\bibitem{Arveson} W.\ Arveson, {\em Operator algebras and measure preserving automorphisms.} Acta Math.\ {\bf 118} (1967), 95--109.

\bibitem{ArvesonJosephson} W.\ Arveson and K.\ Josephson, {\em Operator algebras and measure preserving automorphisms. II.} J.\ Funct.\ Anal.\ {\bf 4} (1969), 100--134.

\bibitem{BrownOzawa} N.\ Brown and N.\ Ozawa, {\em $C^*$-algebras and finite-dimensional approximations.} Graduate Studies in Mathematics, 88. American Mathematical Society, Providence, RI, 2008.

\bibitem{Choda} H.\ Choda, {\em A correspondence between subgroups and subalgebras in a discrete $C^*$-crossed product.} Math.\ Japonica {\bf 24} (1979), 225--229.

\bibitem{CuntzLi} J.\ Cuntz and X.\ Li, {\em The regular $C^*$-algebra of an integral domain.} Quanta of maths, 149--170 {\bf Clay Math.\ Proc.\ 11} Amer.\ Math.\ Soc.\,  Providence RI, 2010.  

\bibitem{DavidsonKatsoulis} K.\ Davidson and E.\ Katsoulis, {\em Isomorphisms between topological conjugacy algebras.} J.\ Reine Angew.\ Math.\ {\bf 621} (2008), 29--51.

\bibitem{Simple} K.\ Davidson and E.\ Katsoulis, {\em Semicrossed products of simple $C^*$-algebras.} Math.\ Ann.\ {\bf 342} (2008), 515--525.

\bibitem{multivariable} K.\ Davidson and E.\ Katsoulis, {\em Operator algebras for multivariable dynamics.} Mem.\ Amer.\ Math.\ Soc.\ {\bf 209} (2011), no.\ 982.

\bibitem{dilation} K.\ Davidson and E.\ Katsoulis, {\em Dilation theory, commutant lifting, and semicrossed products.} Doc.\ Math.\ {\bf 16} (2011), 781--868

\bibitem{DFK} K.\ Davidson, A.\ Fuller, and E.\ Kakariadis, {\em Semicrossed products of operator algebras by semigroups.} Memoirs Amer.\ Math.\ Soc.\ {\bf 239} (201X). 

\bibitem{Duncan} B.\ Duncan, {\em Operator algebras associated to integral domains.} New York J.\ Math.\ {\bf 19} (2013), 39--50.

\bibitem{DuncanPeters} B.\ Duncan and J.\ Peters, {\em Operator algebras and representations from commuting semigroup actions.} J.\ Operator Theory {\bf 74} (2015), 23--43.

\bibitem{Fuller} A.\ Fuller, {\em Nonself-adjoit semicrossed products by abelian semigroups.} Canad.\ J.\ Math.\ {\bf 65} (2013), 768--782.

\bibitem{KK} E.\ Kakariadis and E.\ Katsoulis, {\em Isomorphism invariants for multivariable $C^*$-dynamics.} J.\ Noncommut.\ Geom.\ {\bf 8} (2014), 771--787.

\bibitem{LOP} M.\ Landstad, D.\ Olesen, and G.\ Pedersen, {\em Towards a Galois theory for crossed products of $C^*$-algebras.} Math.\ Scand.\ {\bf 43} (1978), 311--321.

\bibitem{Peters} J.\ Peters, {\em Semicrossed products of $C^*$-algebras.} J.\ Funct.\ Anal.\ {\bf 59} (1984), 498--534.

\bibitem{Envelope} J.\ Peters, {\em The $C^*$-envelope of a semicrossed product and nest representations.} Operator structures and dynamical systems, 197--215 {\bf Contemp.\ Math., 503}, Amer.\ Math.\ Soc., Providence RI, 2009. 

\bibitem{Roman} S.\ Roman, {\em Advanced linear algebra, third edition.} Graduate Texts in Mathematics, 135. Springer, New York, NY, 2008.

\end{thebibliography}

\end{document}